\date{}
\title{Visibility to infinity in the hyperbolic plane,\\
despite obstacles}
\author{Itai Benjamini\thanks{Departments of Mathematics, The Weizmann Institute, Rehovot, Israel 76100. E-mail: {\tt itai.benjamini@weizmann.ac.il}} \and Johan Jonasson\thanks{Department of Mathematical Sciences,
    Division of Mathematics, Chalmers University of
    Technology, S-41296 G\"oteborg, Sweden. E-mail: {\tt
      jonasson@math.chalmers.se}. Research partially supported by the Swedish
    Natural Science Research Council.} \and Oded Schramm\thanks{Microsoft Research, One Microsoft Way, Redmond WA 98052, USA. }
\and Johan Tykesson\thanks{Department of Mathematical Sciences,
    Division of Mathematical Statistics, Chalmers University of
    Technology, S-41296 G\"oteborg, Sweden. E-mail: {\tt
      johant@math.chalmers.se}. Research supported by the Swedish
    Natural Science Research Council.}}

\documentclass[12p]{article}
\usepackage{amsmath}
\usepackage{amsthm}
\usepackage{amsfonts}
\usepackage{graphicx}
\input labelfig.tex
\usepackage{epsf}
\newif\ifhyper\IfFileExists{hyperref.sty}{\hypertrue}{\hyperfalse}
\hyperfalse 
\ifhyper\usepackage[hypertex]{hyperref}\fi

\newif\ifdraft
\drafttrue
\numberwithin{equation}{section}
\numberwithin{figure}{section}

\newtheorem{theorem}{Theorem}
\numberwithin{theorem}{section}

\newtheorem{lemma}[theorem]{Lemma}

\newtheorem{proposition}[theorem]{Proposition}
\newtheorem{conjecture}[theorem]{Conjecture}
\newtheorem{problem}[theorem]{Problem}
\theoremstyle{definition}\newtheorem{remark}[theorem]{Remark}
\newtheorem{question}[theorem]{Question}
\newtheorem{definition}[theorem]{Definition}

\newcommand{\R}{\mathbb{R}}

\newcommand{\Z}{\mathbb{Z}}
\newcommand{\N}{\mathbb{N}}

\def\H{\mathbb{H}}

\def\length{\mathrm{length}}
\def\ev#1{\mathcal{#1}}

\def \eps {\epsilon}

\def \P {{\bf P}}
\def\md{\mid}
\def\Bb#1#2{{\def\md{\bigm| }#1\bigl[#2\bigr]}}

\def\Bs#1#2{{\def\md{\mid}#1[#2]}}
\def\Pb{\Bb\P}
\def\Eb{\Bb\E}

\def\Es{\Bs\E}

\def \p {{\partial}}
\def \E {{\bf E}}

\def\proofof#1{{ \medbreak \noindent {\bf Proof of #1.} }}

\def \_reg {\rightarrow_{\bf reg}}

\def\maxdeg/{\Delta}

\def\area{\mathop{\mathrm{area}}}

\def\length{\mathop{\mathrm{length}}}

\def\gv{\mathrm{gv}}
\def\gc{\mathrm{gc}}
\def\OC{\mathcal B}
\def\VC{\mathcal W}
\def\gZ{\mathcal Z}
\def\vc{\lambda_{\gv}}
\def\barvc{\bar{\lambda}_{\gv}}
\def\oc{\lambda_{\gc}}
\def\baroc{\bar{\lambda}_{\gc}}

\begin{document}
\maketitle


\begin{abstract}
  Suppose that $\gZ$ is a random closed subset of the hyperbolic plane
  $\H^2$, whose law is invariant under isometries of $\H^2$.  We prove
  that if the probability that $\gZ$ contains a fixed ball of radius
  $1$ is larger than some universal constant $p_0<1$, then there is
  positive probability that $\gZ$ contains (bi-infinite) lines.

  We then consider a family of random sets in $\H^2$ that satisfy
  some additional natural assumptions. An example of such a set is the
  covered region in the Poisson Boolean model. Let $f(r)$ be the
  probability that a line segment of length $r$ is contained in such a
  set $\gZ$. We show that if $f(r)$ decays fast enough, then there are
  a.s.\ no lines in $\gZ$. We also show that if the decay of $f(r)$ is
  not too fast, then there are a.s.\ lines in $\gZ$. In the case of
  the Poisson Boolean model with balls of fixed radius $R$ we
  characterize the critical intensity for the a.s.\ existence of lines
  in the covered region by an integral equation.

  We also determine when there are lines in the complement of a Poisson process on the Grassmannian of lines in $\H^2$.
\end{abstract}




\section{Introduction and main results}

In this paper, we are interested in the existence of hyperbolic
half-lines and lines (that is, infinite geodesic rays and bi-infinite
geodesics respectively) contained in unbounded connected components of
some continuum percolation models in the hyperbolic plane. Our first
result is quite general:

\begin{theorem}\label{t.gen}
Let $\gZ$ be a random closed subset of $\H^2$, whose law is invariant under
isometries of $\H^2$, and let $B$ denote some fixed ball of radius $1$
in $\H^2$. There is a universal constant $p_o<1$ such that
if $\Pb{B\subset \gZ}>p_o$, then with positive probability $\gZ$ contains
hyperbolic lines.
\end{theorem}

The first result of this type was proven by Olle
H\"aggstr\"om~\cite{H} for regular trees of degree at least $3$.  That
paper shows that for automorphism invariant site percolation on such
trees, when the probability that a site is open is sufficiently close
to $1$, there are infinite open clusters with positive probability.
This was subsequently generalized to transitive nonamenable graphs
\cite{BLPS}.  The new observation here is that in $\H^2$, one can
actually find lines contained in unbounded components when the
marginal is sufficiently close to $1$. The proof of Theorem
\ref{t.gen} is not too difficult, and is based on a reduction to the
tree case.

We conjecture that Theorem~\ref{t.gen} may be strengthened by
taking $\gZ$ to be open and replacing
the assumption $\Pb{B\subset\gZ}>p_o$ with
$\Eb{\length(B\setminus\gZ)}<\delta$; see Conjecture~\ref{co.length}
and the discussion which follows.

We also obtain more refined results for random sets that satisfy a
number of additional conditions. One example of such a set is the following. Consider a Poisson point process with
intensity $\lambda$ on a manifold $M$. In the {\it Poisson Boolean
  model of continuum percolation} with parameters $\lambda$ and $R$,
balls of radius $R$ are centered around the points of the Poisson
process. One then studies the geometry of the connected components of
the union of balls, or the connected components of the complement. In
particular, one asks for which values of the parameters there are
unbounded connected components or a unique unbounded component.

In the setting of the Poisson Boolean model in the hyperbolic plane, Kahane~\cite{K1,K2}
showed that if $\lambda<1/(2\sinh R)$, then the set of rays from a
fixed point $o\in \H^2$ that are contained in the complement of the
balls is non-empty with positive probability, while if $\lambda\ge
1/(2\sinh R)$ this set is empty a.s.\ Lyons \cite{LY} generalized the
result of Kahane to $d$-dimensional complete simply-connected
manifolds of negative curvature, and in the case of constant negative
curvature also found the exact value of the critical intensity for the a.s.\ existence of lines.

In this note, we are interested in finding not only rays but lines in
the union of balls and/or its complement. We work mostly in the
hyperbolic plane, but raise questions for other spaces as well. Our
proofs cover the results of Kahane as well, but are also valid for a
larger class of random sets.  We remark that it is easy to see that in
$\R^n$, there can never be rays in the union of balls or in the
complement.

Other aspects of the Poisson Boolean model in $\H^2$ have previously
been studied in \cite{T}. For further studies of percolation in the
hyperbolic plane, the reader may consult the papers~\cite{BS1, L}.
In~\cite{CFKP}, an introduction to hyperbolic geometry is found, and
for an introduction to the theory of percolation on infinite graphs
see, for example,~\cite{BS2,LP,HJ}.  \medskip

Let $X=X_\lambda$ be the set of points in a Poisson process of
intensity $\lambda$ in $\H^2$.  Let
$$
\OC:=\bigcup_{x\in X} \overline{B(x,R)}
$$
denote the {\it occupied set}, where
$B(x,r)$ denotes the open ball of radius $r$ centered at $x$.
The closure of the complement
$$
\VC:=\overline{\H^2\setminus\OC}
$$
will be reffered to as the {\it vacant set}.

Let $\oc=\oc(R)$ denote the supremum of the set of $\lambda\ge 0$ such
that for the parameter values $(R,\lambda)$ a.s.\ $\OC$ does not
contain a hyperbolic line.  Let $\baroc$ denote the supremum of the
set of $\lambda\ge 0$ such that the probability that a fixed point
$x\in \H^2$ belongs to a half-line contained in $\OC$ is $0$.
Similarly let $\vc=\vc(R)$ denote the infimum of the set of
$\lambda\ge 0$ such that for the parameter values $(R,\lambda)$ a.s.\
$\VC$ does not contain a hyperbolic line. Finally, let $\barvc$ denote
the infimum of the set of $\lambda\ge 0$ such that the probability
that a fixed point $x\in \H^2$ belongs to a half-line contained in
$\VC$ is $0$. Later, we shall see that $\vc=\barvc$ and $\oc=\baroc$.
Clearly, if $\lambda>\vc$, there are a.s.\ no hyperbolic lines in
$\VC$ and if $\lambda<\oc$ there are a.s.\ no hyperbolic lines in
$\OC$.  Let $f(r)=f_{R,\lambda}(r)$ denote the probability that a
fixed line segment of length $r$ in $\H^2$ is contained in $\VC$.

\begin{theorem}\label{mainthm}
For every $R>0$, we have $0<\oc(R)=\baroc(R)<\infty$,
and the following statements hold at $\oc(R)$.
\begin{enumerate}
\item A.s.\ there are no hyperbolic lines within $\OC$.
\item Moreover, $\OC$ a.s.\ does not contain any hyperbolic ray (half-line).
\item There is a constant $c=c_R>0$, depending only on $R$, such that
\begin{equation}
\label{e.expo}
c\,e^{-r}\le f(r)\le e^{-r}, \qquad \forall r>0\,.
\end{equation}
\end{enumerate}

Furthermore, the analogous statements hold with $\VC$ in place of
$\OC$ (with possibly a different critical intensity).
\end{theorem}

An equation characterizing $\oc$
follows from our results (i.e.,~\eqref{e.aeq} with $\alpha=1$).

The key geometric property allowing for geodesic percolation to occur
for some $\lambda$ is the exponential divergence of geodesics. This
does not hold in Euclidean space. It is of interest to determine which
homogeneous spaces admit a regime of intensities with geodesics
percolating.

With regards to higher dimensions, we show that in hyperbolic space of any dimension $d\le 3$ and for any $(R,\lambda)\in(0,\infty)^2$, there can never be planes contained in the covered or vacant region of the Poisson Boolean model.

We also consider a Poisson process $Y$ on the Grassmannian of lines on $\H^2$. We show that if the intensity of $Y$ is sufficiently small, then there are lines in the complement of $Y$ (when $Y$ is viewed as a subset of $\H^2$), which means that $Y$ is not connected. On the other hand, if the intensity is large enough, then the complement of $Y$ contains no lines, which means that $Y$ is connected. At the critical intensity, $Y$ is connected.

\medskip

Our paper ends with a list of open problems.

\section{Lines appearing when the marginal is large}

\begin{figure}
\SetLabels
\endSetLabels
\centerline{\epsfysize=2.2in%
\AffixLabels{%
\epsfbox{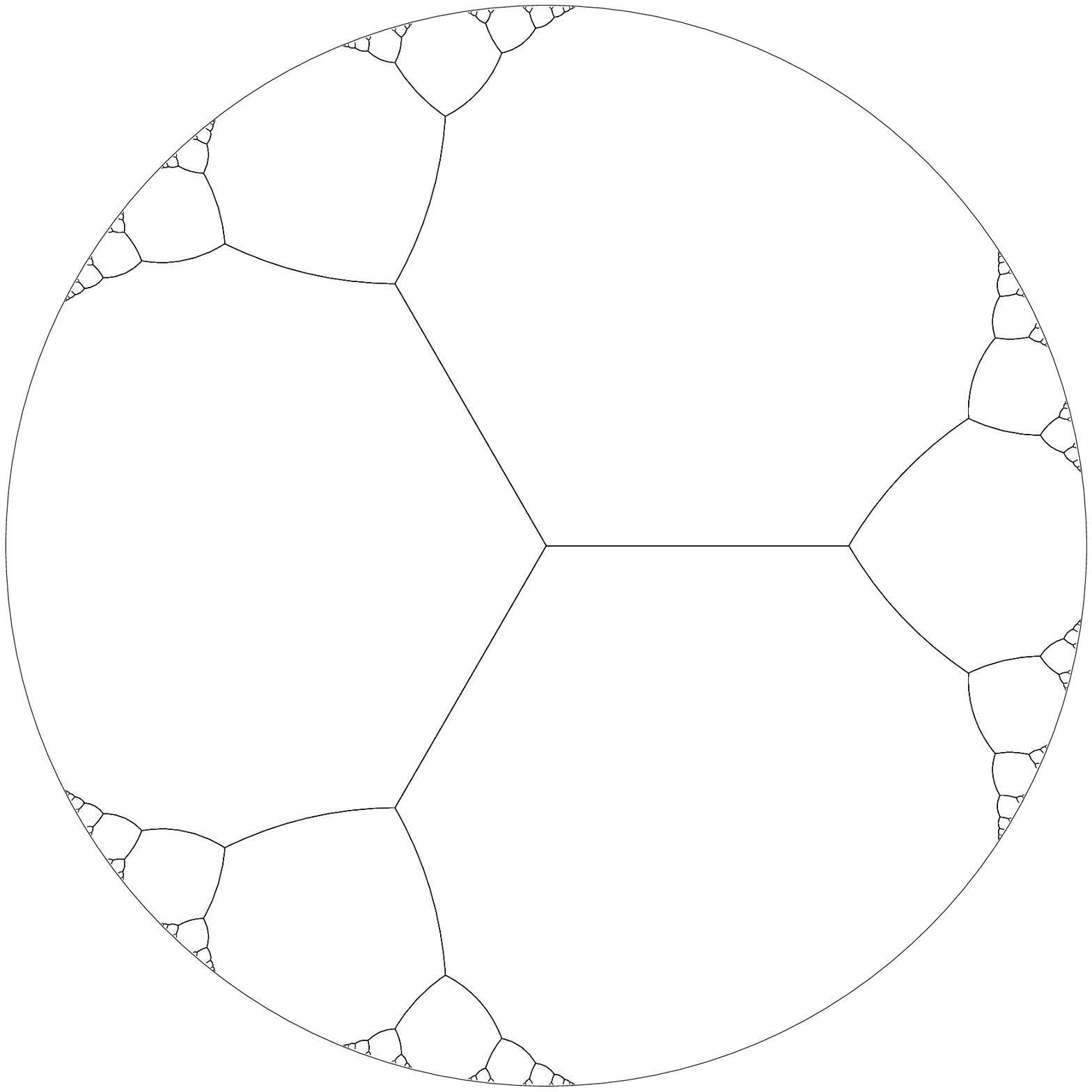}\qquad%
\epsfysize=2.2in\epsfbox{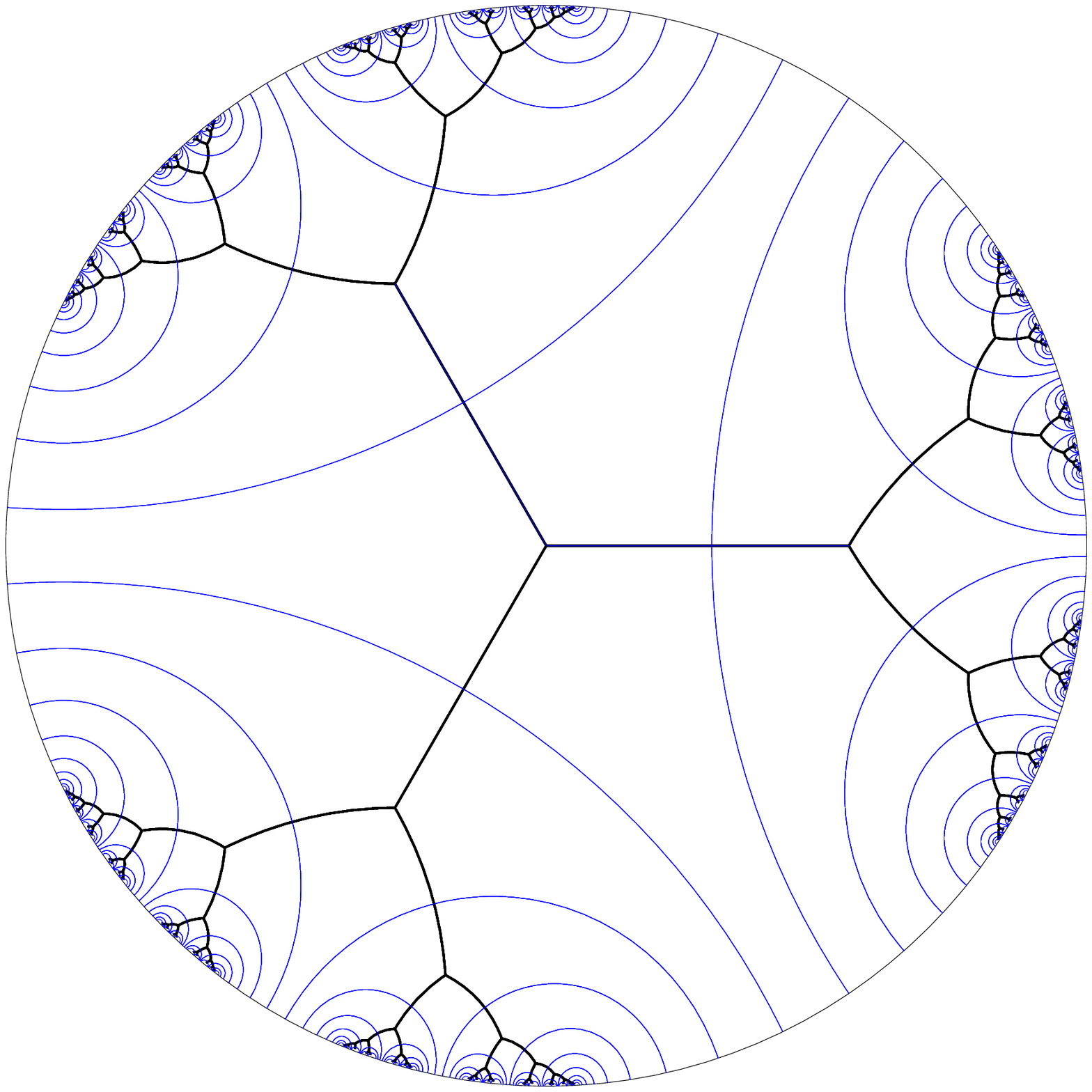}}
}
\begin{caption} {\label{f.hyptree}
A tree embedded in the hyperbolic plane,
in the Poincar\'e disk model. On the right appears the tree together with
some of its lines of symmetry.}
\end{caption}
\end{figure}

The proof of Theorem~\ref{t.gen} is based on a reduction to the tree
case.  We will need the following construction of a tree embedded in
$\H^2$, which is illustrated in Figure~\ref{f.hyptree}.  (This
construction should be rather obvious to the readers who are
proficient in hyperbolic geometry.)  Consider the hyperbolic plane in
the Poincar\'e disk model.  Let $o\in\H^2$ correspond to the center of
the disk.  Let $A_0$ be an arc on the unit circle of length smaller
than $2\,\pi/3$.  Let $A_j$ denote the rotation of $A_0$ by
$2\,\pi\,j/3$; that is $A_j:= e^{2\pi j/3} A_0$, $j=1,2$.  Let $L_j$,
$j=0,1,2$, denote the hyperbolic line whose endpoints on the ideal
boundary $\p\H^2$ are the endpoints of $A_j$.  Let $\Gamma$ denote the
group of hyperbolic isometries that is generated by the reflections
$\gamma_0,\gamma_1$ and $\gamma_2$ in the lines $L_0,L_1$ and $L_2$,
respectively.  If $w=(w_1,w_2,\dots,w_n)\in \{0,1,2\}^n$, then let
$\gamma_w$ denote the composition $\gamma_{w_1}\circ
\gamma_{w_2}\circ\cdots\circ \gamma_{w_n}$.  We will say that $w$ is
\emph{reduced} if $w_{j+1}\ne w_j$ for $j=1,2,\dots,n-1$.  A simple
induction on $n$ then shows that $\gamma_w (o)$ is separated from $o$
by $L_{w_1}$ when $w$ is reduced and $n>0$.  In particular, for
reduced $w\ne ()$, we have $\gamma_w(o)\ne o$ and $\gamma_w\ne
\gamma_{()}$.  Clearly, every $\gamma_w$ where $w$ has $w_j=w_{j+1}$
for some $j$ is equal to $\gamma_{w'}$ where $w'$ has these two
consecutive elements of $w$ dropped.  It follows that $\Gamma$ acts
simply and transitively on the orbit $\Gamma o$.  (\lq\lq
Simply\rq\rq\ means that $\gamma v=v$ where $\gamma\in\Gamma$ and
$v\in\Gamma o$ implies that $\gamma$ is the identity.)  Now define a
graph $T$ on the vertex set $\Gamma o$ by letting each $\gamma(o)$ be
connected by edges to the three points $\gamma \circ \gamma_j (o)$,
$j=0,1,2$. Then $T$ is just the $3$-regular tree embedded in the
hyperbolic plane. In fact, this is a Cayley graph of the group
$\Gamma$, since we may identify $\Gamma$ with the orbit $\Gamma o$.
(One easily verifies that $\Gamma$ is isomorphic to the free product
$\Z_2*\Z_2*\Z_2$.)

We will need a few simple properties of this embedding of the 3-regular tree in
$\H^2$.
It is easy to see that every simple path
$v_0, v_1,\dots$ in $T$ has a unique limit point on the ideal boundary $\p\H^2$.
(Figure~\ref{f.hyptree} does not lie.)
Moreover, if $v_0=o$ and $v_1= \gamma_j(o)$, then the limit point will
be in the arc $A_j$. If $(v_j:j\in\Z)$ is a bi-infinite simple path in
$T$ with $v_0=o$, then its two limit points on the ideal boundary will
be in two different arcs $A_j$.  Hence, the distance from $o$ to the
line in $\H^2$ with the same pair of limit points on $\p\H^2$ is
bounded by some constant $R$, which does not depend on the path
$(v_j:j\in\Z)$.  Invariance under the group $\Gamma$ now shows that
for every bi-infinite simple path $\beta$ in $T$, the hyperbolic line
$L_\beta$ joining its limit points passes within distance $R$ from
each of the vertices of $\beta$. It follows that there is some
constant $R'>0$ such that $L_\beta$ is contained in the
$R'$-neighborhood of the set of vertices of $\beta$.

We are now ready to prove our first theorem.
\proofof{Theorem~\ref{t.gen}}
We use the above construction of $T$, $\Gamma$ and the constant $R'$.
Given $\gZ$, let $\omega\subset V(T)$ denote the set
of vertices $v\in V(T)$ such that the ball $B(v,R')$ is contained in
$\gZ$. Then $\omega$ is a (generally dependent) site percolation on $T$ and its
law is invariant under $\Gamma$.
Set $q:=\Pb{o\in\omega}$.
By~\cite{BLPS}, there is some $p_0\in(0,1)$ such that if $q\ge p_0$, then
$\omega$ has infinite connected components with positive probability.
(We need to use~\cite{BLPS}, rather than~\cite{H}, since the group $\Gamma$ is
not the full automorphism group of $T$.)
Let $N$ be the number of balls of radius $1$ that are sufficient to cover
$B(o,R')$.
Now suppose that $\Pb{B(o,1)\subset\gZ}> 1- (1-p_0)/(2N)$.
Then a sum bound implies that $q> (p_0+1)/2$.
Therefore, if we intersect $\omega$ with an independent
Bernoulli site percolation with marginal $p>(p_0+1)/2$, the resulting
percolation will still have infinite components with positive probability,
by the same argument as above. Thus, we conclude that with positive
probability $\omega$ has infinite components with more than one end
and therefore also bi-infinite simple paths. The line determined by
the endpoints on $\p\H^2$ of such a path will be contained in
$\gZ$, by the definition of $R'$. The proof is thus complete.
\qed

\section{Lines in well-behaved percolation}\label{well}

The proofs of the statements in Theorem~\ref{mainthm}
concerning $\OC$ are essentially the same as the proofs
concerning $\VC$. We therefore find it worthwhile to employ an
axiomatic approach, which will cover both cases.

\medskip

\begin{definition}\label{d.well}
In the following, we fix a closed disk $B\subset \H^2$ of radius $1$.
A \emph{well-behaved percolation} on $\H^2$ is a random closed subset
$\gZ\subset\H^2$ satisfying the following assumptions.
\begin{enumerate}
\item The law of $\gZ$ is invariant under isometries of $\H^2$.
\item The set $\gZ$ satisfies positive correlations; that is,
for every pair $g$ and $h$ of bounded increasing measurable functions
of $\gZ$, we have $$\Eb{g(\gZ)\,h(\gZ)}\ge \Eb{g(\gZ)}\,\Eb{h(\gZ)}.$$
\item There is some $R_0<\infty$ such that $\gZ$ satisfies
  independence at distance $R_0$, namely, for every pair of subsets
  $A,A'\subset \H^2$ satisfying $\inf\{d(a,a'):a\in A,a'\in A'\}\ge
  R_0$, the intersections $\gZ\cap A$ and $\gZ\cap A'$ are
  independent.
\item The expected number $m$ of connected components of $B\setminus
  \gZ$ is finite.
\item The expected length $\ell$ of $B\cap \partial\gZ$ is finite.
\item $p_0:= \Pb{B\subset\gZ}>0$.
\end{enumerate}
\end{definition}

Invariance under isometries implies that $m$, $\ell$ and $p_0$ do not depend on the choice of $B$.
We say that $\gZ$ is $\Lambda$-well behaved, if it is well-behaved and
$p_0,m^{-1},\ell^{-1},R_0^{-1}>\Lambda$.
Many of our estimates below can be made to depend only on $\Lambda$.
In the following, we assume that $\gZ$ is $\Lambda$-well behaved,
where $\Lambda>0$, and use $O(g)$ to denote any quantity bounded by $c\,g$, where
$c$ is an arbitrary constant that may depend only on $\Lambda$.

\medskip

If $x,y\in\H^2$, let $[x,y]_s$
denote the union of all line segments $[x',y']$ where $d(x,x')< s$
and $d(y,y')<s$. Let $A(x,y,s)$ be the event that there is some
connected component of $\gZ\cap [x,y]_s$ that intersects
$B(x,s)$ as well as $B(y,s)$,
and let $Q(x,y,s)$ be the event that $[x,y]_s\subset \gZ$.

\begin{lemma}\label{lemma a}
There is a constant $c=c(\Lambda)<\infty$, which depends only on $\Lambda$,
such that for all $x,y\in\H^2$ satisfying $d(x,y)\ge 4$ and for all $\eps>0$
\begin{equation}
\label{e.a}
  \Pb{Q(x,y,\eps)} > (1-c\,\eps)\, \Pb{ A(x,y,\eps)}.
\end{equation}
\end{lemma}

\begin{proof}
Observe that the expected minimal number of disks of small radius $\eps$
that are needed to cover $\partial\gZ\cap B$ is $O(\ell/\eps)$.
It follows by invariance that
\begin{equation}
\label{e.bdry}
  \Pb{B(x,\eps)\cap \p\gZ\ne\emptyset}= O(\eps)\,\ell=O(\eps)
\end{equation}
holds for $x\in\H^2$.

Let $\gamma:\R\to\H^2$ denote a hyperbolic line parameterized by arclength,
and let $L_t$ denote the hyperbolic line through $\gamma(t)$ which is
orthogonal to $\gamma$.
Set
$$
g(r,s):=\Pb{A\bigl(\gamma(0),\gamma(r),s\bigr)\setminus Q\bigl(\gamma(0),\gamma(r),s\bigr)}.
$$
By invariance, we have $\Pb{A(x,y,s)\setminus Q(x,y,s)}=g\bigl(d(x,y),s\bigr)$.

Set $B:= B\bigl(\gamma(0),1\bigr)$.
Fix some $\eps\in(0,1/10)$.
Let $S_j$ denote the intersection of $B$ with the open strip between $L_{2j\eps}$ and
$L_{2(j+1)\eps}$, where $j\in J:=\N\cap[0,\eps^{-1}/10]$.
Let $x_j$ and $y_j$ denote the two points in $L_{(2j+1)\eps}\cap \p B$.
Let $J_1$ denote the set of $j\in J$ such that
$S_j$ is not contained in $\gZ$ but there is a connected component of
$\gZ\cap S_j$ that joins the
two connected components of $S_j\cap\partial B$.
Observe that the number of connected components of $B\setminus\gZ$ is
at least $|J_1|-1$. Hence $\Eb{|J_1|}\le m+1$.
Let $J_2$ denote the set of $j\in J$ such that
$A(x_j,y_j,\eps)\setminus Q(x_j,y_j,\eps)$ holds.
Note that if $j\in J_2\setminus J_1$, then
$\p \gZ$ is within distance $O(\eps)$ from
$x_j\cup y_j$.
Therefore, $\Pb{j\in J_2\setminus J_1} =O(\eps)\,\ell$
holds for every $j\in J$, by~\eqref{e.bdry}. Consequently,
$$
\Eb{|J_2|}\le \Eb{|J_2\setminus J_1|}+
\Eb{|J_1|}
\le
O(\eps)\,\ell\,|J|+m+1
=O(1)\,.
$$
Thus, there is at least one $j=j_\eps\in J$ satisfying
\begin{equation}
\label{e.xyeps}
\begin{aligned}
&
\Pb{A(x_j,y_j,\eps)\setminus Q(x_j,y_j,\eps)}
=\Pb{j\in J_2}
\\& \qquad\qquad\qquad
 \le O(1)/|J|=O(\eps)\,.
\end{aligned}
\end{equation}
Set $r_\eps:=d(x_{j_\eps},y_{j_\eps})$, and note that $r_\eps\in(1,2]$.
Now suppose that $x,y\in\H^2$ satisfy $d(x,y)=2$.
Let $x_0$ be the point in $[x,y]$ at distance $r_\eps$ from $y$, and
let $y_0$ be the point in $[x,y]$ at distance $r_\eps$ from $x$.
Observe that $A(x,y,\eps)\subset A(x_0,y,\eps)\cap A(x,y_0,\eps)$.
Moreover, since $[x,y]_\eps\subset [x,y_0]_\eps\cup [x_0,y]_\eps$,
we have $Q(x,y,\eps)\supset Q(x,y_0,\eps)\cap Q(x_0,y,\eps)$.
Thus,
$$
A(x,y,\eps)\setminus Q(x,y,\eps)\subset
\bigl(A(x_0,y,\eps)\setminus Q(x_0,y,\eps)\bigr)
\cup
\bigl(A(x,y_0,\eps)\setminus Q(x,y_0,\eps)\bigr)
$$
and therefore~\eqref{e.xyeps} and invariance gives
\begin{equation}
\label{e.xy2}
g(2,\eps) \le 2\,
\Pb{A(x_{j_\eps},y_{j_\eps},\eps)\setminus Q(x_{j_\eps},y_{j_\eps},\eps)}
=O(\eps)\,.
\end{equation}
The same argument shows that
\begin{equation}
\label{e.qadd}
g(r',\eps)\le 2\,g(r,\eps),\qquad\text{if }2\le r<r'\le 2\,r\,.
\end{equation}

We will now get a bound on $g(2\,k,\eps)$ for large $k\in\N$.
For $j\in[k]:=\N\cap[0,k]$, let $r_j$ be the distance from
$\gamma(2\,j)$ to the complement of $[\gamma(0),\gamma(2\,k)]_\eps$.
Let $A_j:=A\bigl(\gamma(2\,j),\gamma(2\,j+2),r_j\vee r_{j+1}\bigr)$,
$Q_j:=Q\bigl(\gamma(2\,j),\gamma(2\,j+2),r_j\vee r_{j+1}\bigr)$,
where $j\in[k-1]$.
Also set $\bar A:= A\bigl(\gamma(0),\gamma(2\,k),\eps\bigr)$.
Then
$$
Q\bigl(\gamma(0),\gamma(2\,k),\eps\bigr)\supset\bigcap_{j=0}^{k-1}
Q_j\,.
$$
Hence,
\begin{equation}
\label{e.gsum}
g(2\,k,\eps) \le \sum_{j=0}^{k-1} \Pb{\bar A\setminus Q_j}.
\end{equation}
We now claim that
\begin{equation}
\label{e.AQ}
\Pb{\bar A\setminus Q_j} =O(1)\,\Pb{\bar A}\,\Pb{A_j\setminus Q_j},
\end{equation}
where the implied constant depends only on $p_0$ and $R_0$.
Let $j':=\lfloor j-R_0/2-2\rfloor$ and $j'':=\lceil j+R_0/2+3\rceil$.
Suppose first that $j'>0$ and $j''<k$.
Let $\bar A'(j)$ denote the event that $\gZ\cap[\gamma(0),\gamma(2\,k)]_\eps$
contains a connected component that intersects
both $B\bigl(\gamma(0),\eps\bigr)$ and
$B\bigl(\gamma(2\,j'),\eps\bigr)$, and let
$\bar A''(j)$ denote the event that $\gZ\cap[\gamma(0),\gamma(2\,k)]_\eps$
contains a connected component that intersects
both $B\bigl(\gamma(2\,j''),\eps\bigr)$ and $B\bigl(\gamma(2\,k),\eps\bigr)$.
Then $\bar A\subset \bar A'(j)\cap \bar A''(j)\cap A_j$.
Independence at distance $R_0$ therefore gives
$$
\Pb{\bar A\setminus Q_j} \le
\Pb{\bar A'(j)\cap \bar A''(j)}\,\Pb{A_j\setminus Q_j}.
$$
Now note that the fact that $\gZ$ satisfies positive correlations shows
that $$\Pb{ \bar A'(j)\cap \bar A''(j)}\le O(1)\,\Pb{\bar A},$$
where the implied constant depends only on $R_0$ and $p_0$.
Thus, we get~\eqref{e.AQ} in the case that $j'>0$ and $j''<k$.
The general case is easy to obtain (one just needs to drop
$\bar A'(j)$ or $\bar A''(j)$ from consideration).
Now,~\eqref{e.gsum} and~\eqref{e.AQ} give
\begin{equation}
\label{e.gg}
g(2\,k,\eps)\le O(1)\, \Pb{\bar A}\sum_{j=0}^{k-1} g(2,r_j\vee r_{j+1})\,.
\end{equation}
Note that there is a universal constant $a\in(0,1)$
such that $ r_j\le a^{|j|\wedge|k-j|}\,\eps$.
(This is where hyperbolic geometry comes into play.)
Hence, we get by~\eqref{e.xy2} and~\eqref{e.gg} that
$g(2\,k,\eps)\le O(1) \Pb{\bar A} \eps$,
where the implied constant may depend on $\ell,m,R_0$ and $p_0$.
This proves~\eqref{e.a} in the case where $d(x,y)$ is divisible by $2$.
The general case follows using~\eqref{e.qadd} with
$r'=d(x,y)$ and $r=2\,\lfloor r'/2\rfloor$.
\qed
\end{proof}
\medskip

Let $f(r)$ denote the probability that a fixed line segment of length
$r$ is contained in $\gZ$. Clearly,
$$
  \Pb{Q(x,y,s)} \le f(\mbox{length}[x,y]) \le \Pb{A(x,y,s)},
$$
and Lemma~\ref{lemma a} shows that for s sufficiently small the upper and lower
bounds are comparable.

\begin{lemma}\label{l.b}
There is a unique $\alpha\ge0$ (which depends on
the law of $\gZ$) and some $c(\Lambda)>0$
(depending only on $\Lambda$) such that
\begin{equation}
\label{e.b}
c\,e^{-\alpha r}\le f(r)\le e^{-\alpha r}
\end{equation}
holds for every $r\ge 0$.
\end{lemma}

\begin{proof}
Since the uniqueness statement is clear, we proceed to prove existence.
Positive correlations imply that
\begin{equation}
\label{e.supmul}
f(r_1+r_2)\ge f(r_1)\,f(r_2)\,,
\end{equation}
that is, $f$ is supermultiplicative.
Therefore, $-\log f(r)$ is subadditive, and Fekete's Lemma says that we must have
$$
\alpha:=\lim_{r\to\infty} \frac{-\log f(r)}r = \inf_{r>0} \frac{-\log f(r)}r\,.
$$
Since for every $r$ we have $\alpha\le -\log\bigl(f(r)\bigr)/r$,
the right inequality in~\eqref{e.b} follows.

On the other hand, if we fix some $R>R_0$, then independence at
distance larger than $R_0$ gives
$$
f(r_1)\,f(r_2)
\ge
f(r_1+R+r_2)
\overset{\eqref{e.supmul}}\ge f(r_1+r_2)\,f(R)\,.
$$
Dividing by $f(R)^2$, we find
that the function $r\mapsto f(r)/f(R)$ is submultiplicative.
Thus, by Fekete's lemma again,
$$
\lim_{r\to\infty} \frac{\log\bigl(f(r)/f(R)\bigr)}{r} =
\inf_{r>0} \frac{\log\bigl(f(r)/f(R)\bigr)}{r}\,.
$$
The left hand side is equal to $-\alpha$, and we get for every
$r>0$
$$
-\alpha \le \frac{\log\bigl(f(r)/f(R)\bigr)}{r}\,.
$$
By positive correlations, there is some $c=c(\Lambda)>0$ such that
$f(R)\ge c$, which implies the left inequality in~\eqref{e.b}.  \qed
\end{proof}
\medskip

\begin{lemma}\label{l.c}
If $\alpha\ge 1$ (where $\alpha$ is defined in Lemma~\ref{l.b}),
then a.s.\ there are no half-lines contained in $\gZ$.
\end{lemma}

\begin{proof}
  Fix a basepoint $o\in\H^2$.
  Let $s=(2c)^{-1}$, where $c$ is the constant in~\eqref{e.a}.
  Then
\begin{equation}\label{e.AQf}
\Pb{A(x,y,s)}/2\le
\Pb{Q(x,y,s)} \le
f\bigl(d(x,y)\bigr) \le e^{-d(x,y)}
\end{equation}
holds for every $x,y\in\H^2$ satisfying $d(x,y)\ge 4$.
 For every integer $r\ge 4$ let $V(r)$ be a
  minimal collection of points on the circle $\partial B(o,r)$ such that the
  disks $B(z,s)$ with $z\in V$ cover that circle. Let $X_r$ be the set of points
  $z\in V(r)$ such that $A(o,z,s)$ holds. By (\ref{e.AQf})
\begin{equation}
\label{e.Xr}
\Eb{|X_r|} \le 2\,|V(r)|\,f(r) = O(1)\,s^{-1}\,\length\bigl(\partial B(o,r)\bigr)\, e^{-r}
= O(1)
\,,
\end{equation}
since we are treating $s$ as a constant and the length of $\partial B(o,r)$ is $O(e^r)$.

The rest of the argument is quite standard, and so we will be brief.
By~\eqref{e.Xr} and Fatou's lemma, we have
$\liminf_{r\to\infty}|X_r|<\infty$ a.s.
  Now fix some large $r$ and let $r'\in\N$ satisfy $r'>r+R_0+2$.
Since $\gZ\setminus B(o,r+R_0+1)$ is independent from $\gZ\cap
B(o,r)$, positive correlations implies that
\begin{equation}\label{e.Xr2}
\Pb{X_{r'}=\emptyset\md \gZ\cap B(o,r)} \ge p^{|X_r|},
\end{equation}
where $p>0$ is a constant (which we allow to depend on the law of
$\gZ$).  Since $\liminf_{r\to\infty} |X_r|<\infty$ a.s., it follows by~\eqref{e.Xr2}
that $\inf_r |X_r|=0$ a.s., which means that $\max\{r:
X_r\ne\emptyset\}<\infty$ a.s.  Therefore, a.s.\ there is no half-line
that intersects $B(o,s)$.  Since $\H^2$ can be covered by a countable
collection of balls of radius $s$, the lemma follows.  \qed
\end{proof}
\medskip

\begin{lemma}\label{l.d}
Suppose that $\alpha<1$. Then (i) a.s.\ $\gZ$ contains hyperbolic lines,
(ii) for every fixed $x\in\H^2$, there is a positive probability
that $\gZ$ contains a half-line containing $x$, and
(iii) for every fixed point $x$ in the ideal boundary $\p\H^2$
there is a.s.\ a geodesic line passing through
$x$ whose intersection with $\gZ$ contains a half-line.
\end{lemma}

\begin{proof}
  We first prove (ii). Fix some point $o\in \H^2$. Let $A$ denote a
  closed half-plane with $o\in\p A$, and let $I:= A\cap \p B(o,1)$.
  For $r>1$ and $x\in \p B(o,1)$, let $L_r(x)$ denote the line segment
  which contains $x$, has length $r$ and has $o$ as an endpoint.  Set
  $Y_r:=\{x\in I: L_r(x)\subset\gZ\}$, and let $y_r$ denote the length
  of $Y_r$.  Then we have
$$
\Eb{y_r}=\length(I)\,f(r)\,.
$$
The second moment is given by
$$
\Eb{y_r^2} = \int_I\int_I \Pb{x,x'\in Y_r}\,dx\,dx'\,.
$$
Now note that if $r_2>r_1>0$, then the distance from
$L_{r_2}(x')\setminus L_{r_1}(x')$ to $L_{r_2}(x)$ is at least
$d(x,x')\,e^{r_1}/O(1)$.  Consequently, by
independence on sets at distance larger than $R_0$, we have
 $$
\Pb{x,x'\in Y_r} \le f(r)\,f\bigl(r+ \log d(x,x') +O(1)\bigr).
$$
Now applying the above and~\eqref{e.b} gives
\begin{multline*}
\frac{\Eb{y_r^2}}{\Es{y_r}^2}
\le
O(1)\,\int_I\int_I \exp\bigl(-\alpha\,\log d(x,x')\bigr)\,dx\,dx'
\\
=
O(1)\,\int_I\int_I d(x,x')^{-\alpha} \,dx\,dx'
=O(1)\,,
\end{multline*}
since $\alpha<1$.
Therefore, the Paley-Zygmund inequality implies  that
$$
 \inf_{r>1} \Pb{y_r>0} >0\,.
$$
Since $y_r$ is monotone non-increasing, it follows that
$$
\Pb{\forall_{r>1}\;\;y_r>0}>0\,.
$$
By compactness, on the event that $y_r>0$ for all $r>1$ we
have
$\bigcap_{r>1} Y_r\ne \emptyset$.
If $x\in \bigcap Y_r$, then the half-line with endpoint
$o$ passing through $x$ is contained in $\gZ\cap A$.
This proves (ii).

We now prove (i).
Fix $s=1/(2\,c)$, where $c$ is given by Lemma~\ref{lemma a}.
For $x\in \p B(o,1)$ let $z_r(x)$ denote the endpoint of $L_r(x)$ that
is different from $o$ and let
$Y'_r$ be the set of points $x\in I$ such that
$[z,z_r(x)]\subset \gZ$ holds for every $z\in B(o,s)$.
Let $y'_r$ denote the length of $Y'_r$.
Then $Y'_r\subset Y_r$ and therefore $y'_r\le y_r$.
By the choice of $s$, we have $\Eb{y'_r}\ge \Eb{y_r}/2$.
On the other hand, $\Eb{(y'_r)^2}\le\Eb{y_r^2}=O(1)\,\Eb{y_r}^2$.
As above, this implies that with positive probability
$Y'_\infty:=\bigcap_{r>1}Y'_r\ne\emptyset$.
Suppose that $x\in Y'_\infty$.
Let $\tilde x$ denote the endpoint on the ideal
boundary $\p\H^2$ of the half-line starting at $o$ and
passing through $x$. Then for every $z\in B(o,s)$ the
half-line $[z,\tilde x)$ is contained in $\gZ$.
By invariance and positive correlations,
for every $\eps>0$ there is positive
probability that $Y'_\infty$ is within distance $\eps$ from
each of the two points in $\p A\cap I$.
If $x'$ and $x''$ are two points in $Y'_\infty$ that are
sufficiently close to the two points in $\p A\cap I$,
then the hyperbolic line joining the two endpoints
at infinity of the corresponding half-lines through $o$
intersects $B(o,s)$. In such a case,  this line will
be contained in $\gZ$. Thus, we see that for every line $L$
(in this case $\p A$) for every point $o\in L$
and for every $\eps>0$, there
is positive probability that $\gZ$ contains a line passing
within distance $\eps$ of the two points in $\p B(o,1)\cap L$.
Now (i) follows by invariance and by independence at a distance.

The proof of (iii) is similar to the above, and will be omitted.
\qed
\end{proof}
\medskip

\begin{remark}
Let $o\in\H^2$.
Let $Y$ denote the set of points $z$ in in the ideal boundary
$\p \H^2$ such that the half-line $[o,z)$ is contained in $\gZ$.
 It can be concluded from the first and second
moments computed in the proof of Lemma~\ref{l.d} and a standard
Frostman measure argument that the essential supremum of the
 Hausdorff dimension of $Y$ is given by
$$
\|\dim_H(Y)\|_\infty=1-\alpha\,.
$$
It would probably not be too hard to show that
$\dim_H(Y)= 1-\alpha$ a.s.\ on the event that $Y\ne\emptyset$.

A modification of the above arguments shows that  there is
positive probability that $\gZ$ contains a line through
$o$ if and only if $\alpha<1/2$. In case $\alpha<1/2$,
the essential supremum of the Hausdorff dimension of the set of
lines in $\gZ$ through
$o$ is $1-2\,\alpha$.

It should be possible to show that the Hausdorff dimension of the
union of the lines in $\gZ$ is a.s.\ $3-2\,\alpha$ when
$\alpha\in [1/2,1)$.
\end{remark}

\section{Boolean occupied and vacant percolation}

Recall the definition of $\OC$ and $\VC$.
First, we show that $\OC$ and $\VC$ are well-behaved.

\begin{proposition}\label{p.well}
Fix a compact interval $I\subset(0,\infty)$.
Then there is some $\Lambda=\Lambda(I)>0$ such that if $\lambda,R\in I$,
then $\OC$ and $\VC$ are $\Lambda$-well behaved.
\end{proposition}
\begin{proof}
It is well known that $\OC$ and $\VC$ satisfy positive correlations.
For ${\VC}$, $m$ is bounded by the expected number of
points in $X$ that fall in the $R$-neighborhood of $B$.
Observe that each connected component of $\VC\cap B$,
with the possible exception of one, has on its boundary
an intersection point of two circles of radius $R$ centered at points
in $X$. Since the second moment of the number of points in $X$ that
fall inside the $R$-neighborhood of $B$ is finite, it follows
that $m$ is also bounded for $\OC$.
The remaining conditions are easily verified and left to the reader.
\qed
\end{proof}
\medskip

We are now ready to prove one of our main theorems.

\proofof{Theorem~\ref{mainthm}}
We start by considering $\OC$.
Fix some $R\in(0,\infty)$. If we let $\lambda\nearrow \infty$, then
$f(1)\nearrow 1$ and by~\eqref{e.b} $\alpha\searrow 0$.
Thus, Lemma~\ref{l.d} implies that $\oc<\infty$.
(We could alternatively prove this from Theorem~\ref{t.gen}.)
It is also clear that $\oc>0$, since for $\lambda$ sufficiently
small a.s.\ $\OC$ has no unbounded connected component.

Since the constant $c$ in Lemma~\ref{l.b} depends only on $\Lambda$,
that lemma implies that $\alpha$ is continuous in
$(\lambda,R)\in(0,\infty)^2$.  In particular, Lemmas~\ref{l.c}
and~\ref{l.d} show that when $\lambda=\oc(R)$, we have $\alpha=1$ and
that there are a.s.\ no half-lines in $\OC$.  Also, we
get~\eqref{e.expo} from~\eqref{e.b}. Finally, it follows from
Lemma~\ref{l.c} and Lemma~\ref{l.d} (ii) that $\oc=\baroc$. The
proof for $\VC$ is similar.  \qed \medskip

Next, we calculate $\alpha$ for $\OC$ and ${\VC}$.

\begin{lemma}\label{l.alphaVC}
The value of $\alpha$ for line percolation in $\VC$ is given by
$$
\alpha=2\,\lambda\sinh R\,.
$$
\end{lemma}

\begin{proof}
  Consider a line $\gamma:\R\to\H^2$, parameterized by arclength, and
  let $r>0$.  A.s.\, the interval $\gamma[0,r]$ is contained in $\VC$
  if and only if the $R$-neighborhood of the interval does not contain
  any points of $X$.  Let $N$ denote this neighborhood, and let $A$
  denote its area.  Then $f(r)=e^{-\lambda A}$.  For each point
  $z\in\H^2$, let $t_z$ denote the $t$ minimizing the distance from
  $z$ to $\gamma(t)$.  Then $N=N_0\cup N_1 \cup N_2$, where
  $N_0:=\{z\in\H^2: d(z,\gamma(t_z))<R, t_z\in[0,r]\}$, $N_1:=\{z\in
  B(\gamma(0),R) :t_z<0\}$ and $N_2:=\{z\in B(\gamma(r),R):t_z>r\}$.
  Observe that $N_1$ and $N_2$ are two half-disks of radius $R$, so
  that their areas are independent of $r$.  We can conveniently
  calculate the area of $N_0$ explicitly in the upper half-plane model
  for $\H^2$, for which the hyperbolic length element is given by
  $|ds|/y$, where $|ds|$ is the Euclidean length element. We choose
  $\gamma(t)=(0,e^t)$.  Recall that the intersection of the upper
  half-plane with the Euclidean circles orthogonal to the real line
  are lines in this model.  It is easy to see that for
  $z=(\rho\,\cos\theta,\rho\,\sin\theta)$, we have
  $\gamma(t_z)=(0,\rho)$.  Moreover, the distance from $z$ to $\gamma$
  is
$$
\Bigl|\int_{\theta}^{\pi/2} \frac{\rho\,d\psi}{\rho\sin\psi}\Bigr|=
\bigl|\log \tan(\theta/2)\bigr|.
$$
Thus, if we choose $\theta\in(0,\pi/2)$ such that $\tan(\theta/2)=e^{-R}$,
then $N_0$ consists of the set
$\bigl\{(\rho\,\cos\psi,\rho\,\sin\psi):
\rho\in[1,e^r],\psi\in(\theta,\pi-\theta)\bigr\}$.
Thus,
$$
\area(N_0)=
\int_\theta^{\pi-\theta}\int_1^{e^r}\frac{\rho\,d\rho\,d\psi}{\rho^2\sin^2\psi}
=2\,r\cot \theta=
r (\cot{\scriptstyle\frac\theta 2}-\tan{\scriptstyle\frac\theta2})
=
2\,r\sinh R\,.
$$
The result follows.
\qed
\end{proof}
\medskip

An immediate consequence of Lemma~\ref{l.alphaVC} is that
$1/(2\sinh(R))$ is the critical intensity for line percolation in
$\VC$.

\begin{remark} Let $\lambda_{\mathrm{c}}(R)$ be the infimum of the set of
  intensities $\lambda\ge 0$ such that $\OC$ contains unbounded
  components a.s.\ Proposition 4.7 in \cite{T} says that
  $\lambda_{\mathrm{c}}(R)/e^{2R}=O(1)$ as $R\rightarrow \infty$, which means
  $\vc(R)>\lambda_{\mathrm{c}}(R)$ for $R$ large. Theorem 4.1 in \cite{T},
  therefore implies that for $R$ large enough, there are intensities for
  which there are lines in $\VC$, but also infinitely many unbounded
  components in both $\VC$ and $\OC$. On the other hand, since
  $\lambda_{\mathrm{c}}(R)\ge 1/(2\pi(\cosh(2R)-1))$, it follows that for $R$
  small enough, we have $\lambda_{\mathrm{c}}(R)>\lambda_{\gv}(R)$. So for $R$
  small enough, there are no intensities for which lines in the vacant
  region coexist with unbounded components in the covered region.
\end{remark}

\begin{lemma}\label{l.alphaOC}
In the setting of line percolation in $\OC$,  $\alpha$ is the unique solution
of the equation
\begin{equation}
\label{e.aeq}
1=\int_0^{2R} e^{\alpha t}\, H_{R,\lambda}'(t)
\,dt\,,
\end{equation}
where
$$
H_{R,\lambda}(t) := -\exp\left(-4\,\lambda\int_0^{t/2}
\sinh\left(\cosh^{-1}\left({\frac{\cosh R}{\cosh s}}\right)\right)\,ds
\right).
$$
\end{lemma}

\begin{proof}
  Consider a line $\gamma:\R\to\H^2$, parameterized by arclength.
  Recall that $X$ is the underlying Poisson process.  We now derive an
  integral equation satisfied by
$$
f(r) = \Pb{\gamma[0,r]\subset\OC}.
$$
For a point $x$ in the $R$-neighborhood of $\gamma$, let
$u_+(x):=\sup\{s:\gamma(s)\in B(x,R)\}$ and
$u_-(x):=\inf\{s:\gamma(s)\in B(x,R)\}$.  Let $X_0:=\{ x\in
X:u_-(x)<0<u_+(x)\}$.  This is the set of $x\in X$ such that
$\gamma(0)\in B(x,R)$.  Also set
$$
S:=\begin{cases} \inf \{u_+(x):x\in X_0\} & X_0\ne\emptyset,\\
  -\infty & X_0=\emptyset\,.
\end{cases}
$$
Assume that $r\ge 2\,R$.  A.s., if $S=-\infty$, then $\gamma[0,r]$ is
not contained in $\OC$. On the other hand, if we condition on $S=s$,
where $s\in(0,2R)$ is fixed, then $\gamma[0,s)\subset \OC$ and the
conditional distribution of $\gamma[s,r]\cap \OC$ is the same as the
unconditional distribution.  (Of course, $S=s$ has probability zero,
and so this conditioning should be understood as a limit.)  Therefore,
we get
\begin{equation}
\label{e.Scond}
\Pb{\gamma[0,r]\subset\OC\md S} = f(r-S),
\end{equation}
where, of course, $f(\infty)=0$.

Let $G(t):=\Pb{S\in(0,t)}$. Shortly, we will
show that $G(t)=H_{R,\lambda}(t)+1$. But presently, we just assume
that $G'(t)$ is continuous and derive~\eqref{e.aeq} with $G$ in place
of $H$.
Since the probability density for $S$ in $(0,2\,R)$ is given
by $G'(t)$, we get from~\eqref{e.Scond}
\begin{equation}
\label{e.rec}
f(r)=\int_0^{2R} f(r-s)\,G'(s)\,ds\,.
\end{equation}
Suppose that $\beta>0$ satisfies
\begin{equation}
\label{e.beta}
1=\int_0^{2R} e^{\beta s}\,G'(s)\,ds\,.
\end{equation}
Since $\int_0^{2R} G'(s)\,ds=\Pb{S>0}<1$,
continuity implies that there is some such $\beta$.
Suppose that there is some $r>0$ such that
$f(r)\le e^{-\beta r}\,f(2R)$, then
let $r_0$ be the infimum of all such $r$.
Clearly, $r_0\ge 2\,R$.
By the definition of $r_0$ and~\eqref{e.rec}, we get
$$
f(r_0) > \int_0^{2R} e^{-\beta (r_0-s)}\, f(2R)\,G'(s)\,ds
\overset {\eqref{e.beta}}=
e^{-\beta r_0}\,f(2\,R)
\,.
$$
Since $f(r)$ is continuous on $(0,\infty)$, this contradicts
the definition of $r_0$.
A similar contradiction is obtained if one assumes that there is some
$r>0$ satisfying $f(r)\ge e^{-\beta (r-2R)}$.
Hence $e^{-\beta r}\,f(2R)\le f(r)\le e^{-\beta(r-2R)}$,
which gives $\alpha=\beta$.

It remains to prove that $G(t)=H_{R,\lambda}(t)+1$.
Let $Q_t:= B\bigl(\gamma(0),R\bigr)\setminus B\bigl(\gamma(t),R\bigr)$.
Observe that
\begin{equation}
\label{e.Gt}
G(t)= \Pb{X\cap Q_t\ne\emptyset}=1-\Pb{X\cap Q_t=\emptyset}=1-e^{-\lambda\,\area(Q_t)}.
\end{equation}
Hence, we want to calculate $\area(Q_t)$.
For $z\in\H^2$ let $u(z)$ denote the $t\in \R$ that minimizes
$d\bigl(z,\gamma(t)\bigr)$, and let $\phi(t,y)$ denote the point in $\H^2$
satisfying $u(z)=t$ which is at distance $y$ to the left of $\gamma$ if
$y\ge 0$, or $-y$ to the right of $\gamma$ otherwise.
Observe that $\bigl\{z\in B\bigl(\gamma(0),R\bigr): u(z)<-t/2\bigr\}$
is isometric to (see Figure~\ref{f.ballsgamma})
$$
\bigl\{z\in B\bigl(\gamma(t),R\bigr): u(z)<t/2\bigr\}=
\bigl\{z\in B\bigl(\gamma(0),R\bigr): u(z)<t/2\bigr\}\setminus \overline{Q_t}\,.
$$
Therefore,
\begin{equation}
\label{e.areaequal}
\area(Q_t)=\area\bigl\{z\in B\bigl(\gamma(0),R\bigr): u(z)\in[-t/2,t/2]\bigr\}.
\end{equation}
By the hyperbolic Pythagorian theorem, we have
$$
\cosh d\bigl(\gamma(0),\phi(s,y)\bigr) = \cosh s\,\cosh y\,.
$$
Hence, the set on the right hand side of~\eqref{e.areaequal}
is
\begin{equation}
\label{e.newset}
\bigl\{\phi(s,y):s\in[-t/2,t/2], \cosh y\le\cosh R/\cosh s\bigr\}.
\end{equation}
At the end of the proof of Lemma~\ref{l.alphaVC}, we saw that the area
of a set of the form $\bigl\{\phi(s,y):s\in[0,r], |y|\le R\bigr\}$ is
$2\,r\,\sinh R$. Hence, the area of~\eqref{e.newset} (and also the area of
$Q_t$) is given by
$$
\int_{-t/2}^{t/2}2\,\sinh\bigl(\cosh^{-1}(\cosh R/\cosh s)\bigr)\,ds\,.
$$
The result follows by~\eqref{e.beta} and~\eqref{e.Gt}, since $\alpha=\beta$.
\qed
\end{proof}
\medskip

\begin{figure}
\SetLabels
\B(.6*.53)$\gamma(t)$\\
\B(.4*.53)$\gamma(0)$\\
\B(.9*.53)$\gamma$\\
\endSetLabels
\centerline{\epsfysize=1in%
\AffixLabels{%
\epsfbox{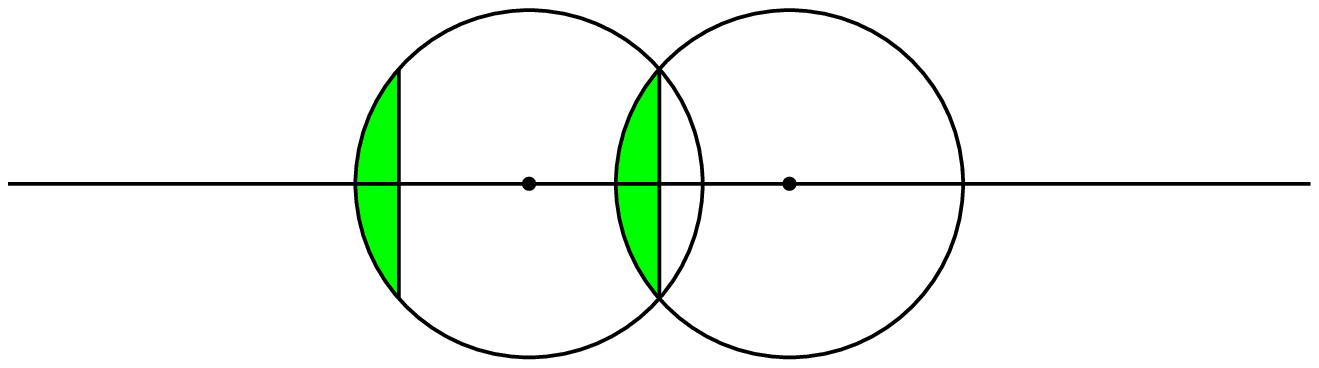}}
}
\begin{caption} {\label{f.ballsgamma}
Calculating the area of $Q_t$.
The set $Q_t$ is the left ball minus the right ball.
The area is calculated by first exchanging the left cap by its \lq\lq shift\rq\rq.}
\end{caption}
\end{figure}

\section{No planes in higher dimensions}

It is natural to ask for high dimensional variants.
Fix some $d\in\N$, $d>2$. Let $\lambda,R>0$.
Let $\OC:=\bigcup_{x\in X} B(x,R)$, where $X$ is
a Poisson point process of intensity $\lambda$ in $\H^d$.
Let $\VC$ be the closure of $\H^d\setminus \OC$.

\begin{proposition} For every $d\in\N\cap[3,\infty)$,
$\lambda,R>0$, a.s.\ there are no $2$-dimensional planes in $\H^d$
that are contained in $\OC$.
Similarly, there are no $2$-dimensional planes in $\H^d$
that are contained in $\VC$.
\end{proposition}

\begin{proof}
  Let $\ev Z$ be $\VC$ or $\OC$.  Fix some $o\in\H^d$, and let $r>0$
  be large. Let $Y_r$ be the set of planes $L$ intersecting the ball
  $B(o,2)$ such that $L\cap B(o,r)$ is contained in the
  $1$-neighborhood of $\ev Z$. If there is a plane $L$ intersecting
  $B(o,1)$ such that $L\cap B(o,r)\subset \ev Z$, then $Y_r$ contains
  the set of planes $L'$ such that the Hausdorff distance between
  $L\cap B(o,r)$ and $L'\cap B(o,r)$ is less than $1$. It therefore
  follows that if there is such an $L$, then the measure of $Y_r$
  (with respect to the invariant measure on the Grassmannian) is at
  least $\exp\bigl(-O(r)\bigr)$. However, if we fix a plane $L$ that
  intersects $B(o,2)$, then $\Pb{L\in Y_r} \le \exp(-c\,e^r)$ for some
  $c=c(d,R,\lambda)>0$, because there are order $e^r$ points in $L\cap
  B(o,r)$ such that the distance between any two is larger than $R+3$.
  This means that the expected measure of $Y_r$ is at most
  $\exp(-c\,e^r)$.  Consequently, the probability that there is some
  plane $L$ intersecting $B(o,1)$ such that $B(o,r)\cap L\subset \ev
  Z$ goes to zero as $r\to\infty$.  \qed
\end{proof}

\section{Connectivity of lines}\label{s.Grassmann}

In this section, we consider a somewhat different model
using a Poisson process on the Grassmannian $\mathbb G$ of lines in $\H^2$.
For this purpose, we first recall the form of an
isometry-invariant measure on $\mathbb G$.
Consider the upper half-plane model for $\H^2$.
Let $\hat\R=\R\cup\{\infty\}=\p\H^2$ denote the boundary
at infinity of $\H^2$.
To each unoriented line $L\subset\H^2$ we may associate
the pair of points of $L$ on the boundary at infinity
$\hat \R$. This defines a bijection between $\mathbb G$
and $$\mathbb M:=\bigl\{\{x,y\}: x,y\in\hat \R,x\ne y\bigr\}\,.$$
(Though we will not use this fact, $\mathbb M$ is an open
M\"obius band, or a punctured projective plane.)
In the following, we often identify $\mathbb M$ and $\mathbb G$
via this bijection, and will not always be careful to distinguish
between them.

The set $\mathbb M$ inherits a locally Euclidean metric coming from
the $2$ to $1$ projection from
$\hat\R\times\hat\R\setminus\text{diagonal}$.
Let $\Phi$ be the measure on $\mathbb M$ whose density at a point
$\{x,y\}\in\mathbb M$ such that $x,y\ne \infty$ with respect to the
Euclidean area measure is
$$d\Phi=\frac{dx\,dy}{(x-y)^2}\,,$$
and $\Phi\bigl(\bigl\{\{x,\infty\}:x\in\R\bigr\}\bigr)=0$.
An isometry $\psi:\H^2\to\H^2$ induces a map $\mathbb G\mapsto\mathbb G$.
In the upper half plane coordinate system,
each such $\psi$ is a transformation of the form $z\mapsto (a\,z+b)/(c\,z+d)$,
with $a,b,c,d\in\R$ and $a\,d-b\,c\ne 0$.
Moreover, $\psi$ extends to a self-homeomorphism of $\hat\R$, and therefore there is
an induced map from $\mathbb M$ to $\mathbb M$.
It is easy to verify using the integration change of variables
formula that $\psi$ preserves the measure $\Phi$.
Hence, $\Phi$ is an isometry-invariant measure on $\mathbb G$.

Let $Y$ be a Poisson point
process of intensity $\lambda$ on $\mathbb G$ with respect to $\Phi$.
With a slight abuse of
notation, we will also write $Y$ for the union of all lines in $Y$,
when viewed as a subset of $\H^2$.
Let $\gZ$ be the complement of $Y$. Observe that a.s.\ $Y$ is connected if
and only if $\gZ$ contains no lines.

\begin{figure}
\begin{center}
\epsfysize=2.2in%
\epsfbox{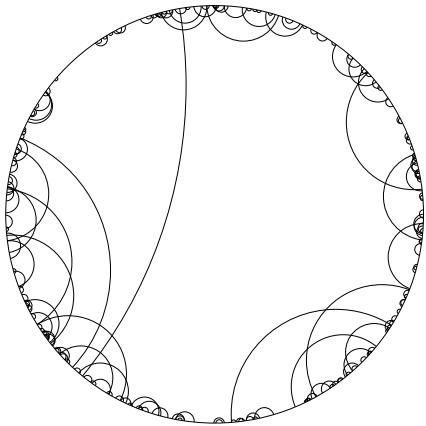}
\end{center}
\caption{\label{grass}A realization of a Poisson process on the Grassmannian of lines in $\H^2$ in the Poincar\'e disk model.}
\end{figure}

\begin{proposition}\label{g.lines}
  If $\lambda\ge 1$ then $\gZ$ contains no lines a.s.\ If $\lambda<1$
  then $\gZ$ contains lines a.s.
\end{proposition}

One motivation for this model comes from long range percolation on
$\Z$. Fix some $c<1$. For each pair $x,y\in \Z$, let there be an edge
between $x$ and $y$ with probability $c$ (independently for different
pairs) if there is a line in $Y$ with one endpoint in $[x,x+1]$ and
the other in $[y,y+1]$. Then a calculation shows that if $\lambda=1$
(which is the critical value), the probability that there is an edge
between $x$ and $y$ is asymptotic to $c/|x-y|^2$ as $|x-y|\to\infty$,
that is, we have
recovered the standard long range percolation model on $\Z$ with critical exponent $2$ (see \cite{BB}). The critical case of long range percolation is not well understood and it might be of interest to further study the connection between it and the line process.

\medskip

Observe that $\gZ$ is not a well-behaved percolation (in the sense of
our definition~\ref{d.well}), since there is
no independence at any distance, and moreover, $\gZ$ is open. Therefore,
several statements in Section~\ref{well} cannot be used directly
to prove Proposition~\ref{g.lines}.
Nevertheless, it is possible to adapt the proofs without much difficulty.

\begin{proof}
First, we calculate $f(r)$. Let $\gamma(t)=(0,e^t)\in\H^2$,
where we think of $\H^2$ in the upper half plane model. Let $A_r$ be the
set of lines that intersect $\gamma[0,r]$. Then it is easy
to see that under the identification
$\mathbb G=\mathbb M$,
 $${A}_r=\bigl\{\{x,y\}\,:\,1\le -xy\le e^{2r}\bigr\}.$$ An easy
calculation shows that $\Phi(A_r)=r$. Therefore,
$f(r)=e^{-\lambda\Phi(A_r)}=e^{-\lambda r}$.

It will now be convenient to use the Poincar\'e disk model.
An unoriented line in $\H^2$ in the Poincar\'e disk model
corresponds to an unordered pair of distinct points on the unit circle,
$x,y\in\p \H^2$. Thus, the measure we have on the Grassmannian
induces a measure on $(\p\H^2)^2$.
By using an isometry between the hyperbolic plane in
the upper half plane model and the hyperbolic plane in the Poincar\'e
disk model it is easy to verify that the density of this
measure is again given locally by $|x-y|^{-2}$ times the product
of the length measure on the circle with itself.

Suppose that $0<\lambda<1$.
For $\theta\in [0,2\,\pi)$ and $r>0$ let $L_r(\theta)$ denote the geodesic
ray of hyperbolic length $r$ started from $0$ whose continuation meets $\p \H^2$
at $e^{i\theta}$.  Let $K_r$ be the set of $\theta\in[0,2\,\pi)$ such
that $L_r(\theta)\subset\gZ$.
Then $\Pb{\theta\in K_r}=e^{-\lambda r}$. To apply the second moment method, we need
to estimate $\Pb{\theta,\theta'\in K_r}$ from above for $\theta,\theta'\in[0,2\,\pi)$.
Suppose first that $\theta'=0$ and $\theta\in[0,\pi]$. Let $L(\theta)$ be the hyperbolic
line $L_\infty(\theta)\cup L_\infty(\theta+\pi)$, which contains $L_r(\theta)$.
The set of pairs $\{x,y\}\in(\p\H^2)^2$ such that the line connecting them
intersects both $L(\theta)$ and $L(\theta')$ is precisely that set of pairs that are separated by
these two lines. The measure of this set is
$$
\int_0^\theta \int_\pi^{\pi+\theta} +
\int_\theta^\pi\int_{\pi+\theta}^{2\pi}
\bigl|e^{i\alpha}-e^{i\beta}\bigr|^{-2}\,d\alpha\,d\beta
=-2\log\frac {\sin\theta}2\,.
$$
The measure of the set of lines that intersect both $L_r(\theta)$
and $L_r(0)$ is bounded by the measure of the set of lines that intersect
both $L_\infty(0)$ and $L_\infty(\theta)$, which is bounded by half
the measure calculated above. The measure of the set of lines
that intersect $L_r(\theta)\cup L_r(0)$ is the sum of the measures
of the lines intersecting each of these segments minus the measure
of the set of lines intersecting both. Thus, it is at least
$$
2\,r+\log \frac{ \sin\theta}2 \,.
$$
This gives
$$
\Pb{\theta,\theta'\in K_r} \le \Bigl(\frac 2{\sin|\theta-\theta'|}\Bigr)^\lambda\, e^{-2\lambda r}\,,
$$
and by symmetry this will also hold if we drop the assumptions that $\theta'=0$ and $\theta\in[0,\pi]$.
Since $\sin^{-\lambda}\theta$ is integrable when $\lambda<1$, this
facilitates the second moment argument, which shows that
$\inf_{r>0}\Pb{K_r\ne\emptyset}>0$. Let $K:=\bigcap_{r>0}\overline K_r$.
Then $\Pb{K\ne\emptyset}=\inf_{r>0}\Pb{K_r\ne\emptyset}>0$, because $K_r\supset K_{r'}$ when
$r'>r$. Now note that a.s.\ $\partial K_r \cap \overline K_{r'}=\emptyset$
when $r'>r$. (The set $\partial K_r$ consists of points in the intersection of $Y$
with the circle of hyperbolic radius $r$ about $0$.)
Hence $K\subset K_r$ holds a.s.\ for each $r>0$.
Thus, with positive
probability there will be some ray $L_\infty(\theta)$ that is contained
in $\gZ$. Clearly, this implies that with positive probability there are at least
three rays corresponding to angles $\theta\in(0,\pi)$. Since the interior
of the convex hull of the union of such rays is in $\gZ$, it follows that
$\gZ$ contains lines with positive probability.
By ergodicity (which is easy to verify), it follows
that there are lines in $\gZ$ a.s.\ when $\lambda<1$.

We now consider the case $\lambda = 1$.
In that case, we can follow the proof of Lemma~\ref{l.c} with
only minor modifications. If $z,z'\in \H^2$ are at distance
$r$ from each other, where $r$ is large, then it is not hard to show
that the set of lines that separate the hyperbolic $1$-ball around
$z$ from the hyperbolic $1$-ball around $z'$ has $\Phi$-measure $r-O(1)$.
Each such line will obviously intersect any line meeting both these balls.
Thus, the probability that there is a line in $\gZ$ meeting both these
balls is $e^{-r+O(1)}$.

The next detail requiring modification is that in the proof
of~\eqref{e.Xr2} independence at a distance was mentioned. However,
it is easy to see that the analog of~\eqref{e.Xr2} holds in our setting.
(Basically, each arc of fixed length on $\p B(o,r)$ can be \lq\lq blocked\rq\rq\
by lines in $Y$ that do not intersect $B(o,r)$ with probability bounded away
from zero.) The argument is then completed as in Lemma~\ref{l.c}. \qed
\end{proof}

\section{Further Problems}

We first consider a generalization of Theorem~\ref{t.gen}.

\begin{conjecture}\label{co.length}
Let $B\subset\H^2$ be some fixed open ball of radius $1$.
There is a constant $\delta>0$ such that if $\gZ\subset\H^2$
is any open random set with isometry-invariant law and
$\Eb{\length(B\setminus\gZ)}<\delta$, then with positive probability
$\gZ$ contains a hyperbolic line.
\end{conjecture}

It is easy to verify that the conjecture implies the theorem. Indeed,
if $\Pb{B\subset\gZ}$ is close to $1$, then one can show that there is
a union of unit circles whose law is isometry invariant, where the
interiors cover the complement of $\gZ$, and where the expected length
of the intersection of the circles with $B$ is small.

\smallskip

Next, we consider quantitative aspects of Theorem~\ref{t.gen}.

\begin{conjecture}\label{co.gen}
Fix some $o\in\H^2$.
For every $r>0$ let $p_r$ be the least $p\in[0,1]$ such that for
every random closed $\gZ\subset\H^2$ with an isometry-invariant law
and $\Pb{B(o,r)\subset\gZ}>p$ there is positive probability that
$\gZ$ contains a hyperbolic line. Theorem~\ref{t.gen} implies that
$p_r<1$ for every $r>0$. We conjecture that
$\limsup_{r\searrow 0} (1-p_r)/r <\infty$.
\end{conjecture}

It is easy to see that $\liminf_{r\searrow 0} (1-p_r)/r >0$;
for example, take a Poisson point process $X\subset\H^2$ with
intensity $\lambda$ sufficiently large and let $\gZ$ be the complement of the
$\eps$-neighborhood of $\bigcup_{x\in X} \partial B(x,1)$,
where $0<\eps<r$.

\begin{problem}
What is $\lim_{r\searrow 0} (1-p_r)/r$?
\end{problem}

The behavior of $p_r$ as $r\to\infty$ seems to be an easier problem,
though potentially of some interest as well.

\medskip

We now move on to problems related to Theorem~\ref{mainthm} and its proof.

\begin{question}
For either $\VC$ or $\OC$,
is there some pair $(\lambda,R)$ for which there is with positive
probability a percolating ray such that every other percolating
ray with the same endpoint at infinity is contained in it?
(Note, such a ray must be exceptional among the percolating rays.)
\end{question}

\begin{question}
Is it true that when $\OC$ (or $\VC$) has a unique infinite connected
component, the union of the lines in $\OC$ (or $\VC$) is connected as well? We believe that there is some pair $(\lambda, R)$ such that $\OC$ contains a unique infinite component but no lines, and we believe the same for $\VC$.
\end{question}

\begin{question}
For which homogenous spaces $\VC$ or $\OC$ a.s.\ contain infinite geodesics
for some parameters $(\lambda,R)$?
\end{question}

Note that since $\H^2\times \R$ contains $\H^2$, it follows that
for every $R$ there is some $\lambda$ such that $\VC$ on $\H^2\times\R$
contains lines within an $\H^2$ slice, and the same holds for $\OC$.

\begin{question}
Let $V$ be the orbit of a point $x\in\H^2$ under a group of isometries
$\Gamma$. Suppose that $V$ is discrete and $\H^2/\Gamma$ is compact.
(E.g., $V$ is a co-compact lattice in $\H^2$.)
Let $\VC_V(R):=\H^2\setminus \bigcup_{v\in V} B(v,R)$, and
let $R^V_c$ denote the supremum of the set of $R$ such that
$\VC_V(R)$ contains uncountably many lines.
Does $\VC_V(R^V_c)$ contain uncountably many lines?
\end{question}

It might be interesting to determine the value of $R^V_c$ for some
lattices $V$.

\begin{problem}
  It is not difficult to adapt our proof to show that in $\H^d$, $d\ge
  2$, for every $R>0$ when $\lambda$ is critical for the existence of
  lines in $\VC$, there are a.s.\ no lines inside $\VC$. This should
  also be true for $\OC$, but we presently do not know a proof.  It
  seems that what is missing is an analog of Lemma~\ref{lemma a}.
\end{problem}

\bigskip \noindent {\bf
Acknowledgements} We are grateful to Wendelin Werner for insights related to
Section~\ref{s.Grassmann}.



\begin{thebibliography}{99}

\bibitem{BB} I. Benjamini and N. Berger, \textit{The diameter of long-range percolation clusters on finite cycles}, Rand. Struct. Alg. {\bf 19}:2 (2001), 102-111.

\bibitem{BLPS} I. Benjamini, R. Lyons, Y. Peres, and O. Schramm,
  \textit{Group-invariant percolation on graphs}, Geom. Funct. Anal.
  {\bf 9} (1999), 29-66.

\bibitem{BS1} I. Benjamini and O. Schramm, \textit{Percolation in the
    hyperbolic plane}, J. Amer. Math. Soc. {\bf 14} (2001), 487-507.

\bibitem{BS2} I. Benjamini and O. Schramm, \textit{Percolation beyond
    ${\mathbb Z}^d$, many questions and a few answers}, Electronic
  Commun. Probab. {\bf 1} (1996), 71-82.

\bibitem{CFKP} J.W. Cannon, W.J. Floyd, R. Kenyon and W.R. Parry.
  Hyperbolic geometry. In \emph{Flavors of geometry}, pp. 59-115,
  Cambridge University Press, 1997.

\bibitem{H} O. H\"aggstr\"om, \textit{Infinite clusters in dependent
    automorphism invariant percolation on trees}, Ann.  Probab. {\bf
    25} (1997), 1423-1436.

\bibitem{HJ} O. H\"aggstr\"om and J. Jonasson, \textit{Uniqueness and
    non-uniqueness in percolation theory}, Probability Surveys {\bf 3}
  (2006), 289-344.

\bibitem{J}J. Jonasson, \textit{Dynamical circle covering with
    homogeneous Poisson updating,} (2007), preprint.

\bibitem{JS} J. Jonasson and J. Steif, \textit{Dynamical models for
    circle covering}, Ann. Probab. To appear.

\bibitem{K1} J.-P. Kahane, \textit{Recouvrements alétoires et théorie
    du potentiel}, Colloq. Math. {\bf 60/61}, 387-411.

\bibitem{K2} J.-P. Kahane, \textit{Produits de poids aléatoires
    indépendants et application}, in \textit{Fractal Geometry and
    Analysis}, ed. J. Belair and S. Dubuc, Kluwer Acad. Publ.,
  Dordrecht, pp. 277-324.

\bibitem{L} S. Lalley, \textit{Percolation Clusters in Hyperbolic
    Tessellations}, Geom. Funct. Anal. {\bf 11} (2001), 971-1030.

\bibitem{LY} R. Lyons, \textit{Diffusion and Random Shadows in
    Negatively-Curved Manifolds}, J. Funct. Anal. {\bf 138} (1996),
  426-448.

\bibitem{LP} R. Lyons with Y. Peres, \textit{Probability on trees and
    networks}, (2007), preprint.
  http://mypage.iu.edu/~rdlyons/prbtree/prbtree.html

\bibitem{MR}R. Meester and R. Roy, \textit{Continuum Percolation},
  Cambridge University Press, New York, 1996.

\bibitem{nilsson} F. Nilsson, \textit{The hyperbolic disc, groups and
    some analysis}, Master's thesis, Chalmers University of
  Technology, 2000.

\bibitem{S} L.A. Shepp, \textit{Covering the circle with random arcs},
  Israel J.  Math. 11 (1972), 328-345.

\bibitem{T} J. Tykesson, \textit{The number of unbounded components in
    the Poisson Boolean model of continuum percolation in hyperbolic
    space}, Electron J. Prob. 12 (2007), 1379-1401.



\end{thebibliography}
\end{document}